\documentclass[english,12pt,reqno]{amsart}

\usepackage{amsmath}
\usepackage{amssymb}
\usepackage{amsthm}
\usepackage{amscd}
\usepackage{babel}
\usepackage[latin1]{inputenc}
\usepackage{graphicx}

\numberwithin{equation}{section}

\catcode`\@=11
\@namedef{subjclassname@2010}{%
  \textup{2010} Mathematics Subject Classification}
\catcode`\@=12

\newtheorem*{teor}{Theorem}
\newtheorem*{coro}{Corollary}

\theoremstyle{definition}

\def\rad{\operatorname{rad}}

\begin{document}

\title{A determinant of generalized Fibonacci numbers}

\author[Christian Krattenthaler]{Christian Krattenthaler$^\dagger$}
\address{Fakult\"at f\"ur Mathematik, Universit\"at Wien,
Nordbergstrasze~15, A-1090 Vienna, Austria.
WWW: {\tt http://www.mat.univie.ac.at/\lower0.5ex\hbox{\~{}}kratt}.}
\author{Antonio M. Oller-Marc\'{e}n}
\address{Centro Universitario de la Defensa\\
Ctra. de Huesca s/n, 50090 Zaragoza (Espa\~{n}a)} 
\email{oller@unizar.es}

\thanks{$^\dagger$Research partially supported by the Austrian
Science Foundation FWF, grants Z130-N13 and S9607-N13,
the latter in the framework of the National Research Network
``Analytic Combinatorics and Probabilistic Number
Theory"}



\begin{abstract}
We evaluate a determinant of generalized Fibonacci numbers, thus
providing a common generalization of several determinant evaluation
results that have previously appeared in the literature, all of them
extending Cassini's identity for Fibonacci numbers.
\end{abstract}

\maketitle

\section{Introduction}

The well-known {\it Fibonacci sequence} is given by
$f_n=f_{n-1}+f_{n-2}$ with $f_0=f_1=1$. Numerous properties of this
sequence are known. 
We refer the reader to the monograph \cite{VAJ} for a wealth of information
on this sequence. 
One of these properties is the so called Cassini identity, given by
$$f_nf_{n+2}-f_{n+1}^2=(-1)^n,$$
which can be written in matrix form as
\begin{equation}
\det\begin{pmatrix} f_n & f_{n+1}\\ f_{n+1} &
f_{n+2}\end{pmatrix}=(-1)^n.\label{casi}\end{equation} 
Miles \cite{MIL} introduced {\it $k$-generalized Fibonacci numbers} 
$f_n^{(k)}$ by
$$f_n^{(k)}=\sum_{i=0}^k f_{n-i}^{(k)},$$
with $f_n^{(k)}=0$ for every $0\leq n\leq k-2$, $f_{k-1}^{(k)}=1$, 
and he gave the following generalization of (\ref{casi}):
\begin{equation}
\det\begin{pmatrix}
f_n^{(k)} & f_{n+1}^{(k)} & \cdots & f_{n+k-1}^{(k)}\\
f_{n+1}^{(k)} & f_{n+2}^{(k)} & \cdots & f_{n+k}^{(k)}\\
\vdots & \vdots & \ddots & \vdots\\
f_{n+k-1}^{(k)} & f_{n+k}^{(k)} & \cdots & f_{n+2k-2}^{(k)}
\end{pmatrix}=(-1)^{\frac{(2n+k)(k-1)}{2}}.\label{casi2}
\end{equation}
More recently, Stakhov \cite{STA} has generalized Cassini's identity
for sequences of the form $f_n=f_{n-1}+f_{n-p-1}$. 

Hoggat and Lind \cite{HOG} consider the so called ``dying rabbit
problem'', previously introduced in \cite{AL1} and studied in
\cite{AL2} or \cite{COH}, which modifies the original Fibonacci
setting by letting rabbits die. In previous work  by one of the
authors \cite{OLL}, the sequence arising in this setting was studied
in detail. For instance, the recurrence relation for this sequence
depends on two parameters $k,\ell\geq 2$ and is given by 
$$C_n^{(k,\ell)}=C_{n-\ell}^{(k,\ell)}+C_{n-\ell-1}^{(k,\ell)}+\cdots+C_{n-k-\ell+1}^{(k,\ell)},$$
where $C_0^{(k,\ell)},\dots,C_{k+\ell-2}^{(k,\ell)}$ are initial values which will be
specified below. It was also proved that, if $r_1,\dots,r_{k+\ell-1}$
are the distinct roots of $g_{k,\ell}(x)=x^{k+\ell-1}-\frac{x^k-1}{x-1}$,
then the general term of the sequence is given by
$C_n^{(k,\ell)}=\displaystyle{\sum_{i=1}^{k+\ell-1} a_ir_i}$, with 
\begin{multline}\label{aes}
a_i=\frac{(-1)^{k+\ell+i-1}}{\displaystyle{\prod_{j>i}(r_j-r_i)
\prod_{j<i}(r_i-r_j)}}\\
\times\left(\sum_{l=0}^{k-2}
    C_l^{(k,\ell)}\frac{r_i^{l+1}-1}{r_i^{l+1}(r_i-1)}+\sum_{l=k-1}^{k+\ell-3}
    C_l^{(k,\ell)}\frac{r_i^k-1}{r_i^{l+1}(r_i-1)}+C_{k+\ell-2}^{(k,\ell)}\right).
\end{multline} 

Given the previous sequence, for every $j\geq 0$ we can define a
matrix $A_{j,k,\ell}$ by 
$$A_{j,k,\ell}=\begin{pmatrix}
C_j^{(k,\ell)} & C_{j+\ell}^{(k,\ell)} & C_{j+\ell+1}^{(k,\ell)} & \dots &
C_{j+k+2\ell-3}^{(k,\ell)}\\ 
C_{j+1}^{(k,\ell)} & C_{j+\ell+1}^{(k,\ell)} & C_{j+\ell+2}^{(k,\ell)} & \dots &
C_{j+k+2\ell-2}^{(k,\ell)}\\ 
\vdots & \vdots & \vdots & \ddots & \vdots\\
C_{j+k+\ell-2}^{(k,\ell)} & C_{j+k+2\ell-2}^{(k,\ell)} & C_{j+k+2\ell-1}^{(k,\ell)} &
\dots & C_{j+2k+3\ell-5}^{(k,\ell)}. 
\end{pmatrix}$$
The main goal of this paper will be to find an explicit expression for
$\det(A_{j,k,\ell})$, thus extending (\ref{casi}) and
(\ref{casi2}). 

\section{Extending Cassini's identity}

Before we proceed, we have to fix our initial conditions. In the
original setting \cite{OLL}, when we start with a pair of rabbits that
become mature $\ell$ months after their birth and die $k$ months after
their matureness, the $k+\ell-1$ initial conditions are given by
$C_0^{(k,\ell)}=\dots=C_{\ell-1}^{(k,\ell)}=1$ and
$C_n^{(k,\ell)}=C_{n-1}^{(k,\ell)}+C_{n-\ell}^{(k,\ell)}$ for every $\ell\leq n\leq
k+\ell-2$. Instead, in what follows we will consider the following
initial conditions: 
\begin{gather*}
\widetilde{C}_0^{(k,\ell)}=1,\\
\widetilde{C}_1^{(k,\ell)}=\dots=\widetilde{C}_{k-1}^{(k,\ell)}=0,\\
\widetilde{C}_k^{(k,\ell)}=\dots=\widetilde{C}_{k+\ell-2}^{(k,\ell)}=1.
\end{gather*}
Note that this change in the initial conditions results only in a
shift of indices. Namely, if $C_n^{(k,\ell)}$ denotes the original
sequence and $\widetilde{C}_n^{(k,\ell)}$ denotes the sequence given by
the same recurrence relation and these new initial conditions, then for every
$n\geq 0$ we have
$$C_n^{(k,\ell)}=\widetilde{C}_{n+k+1}^{(k,\ell)}.$$
Thus, if $\widetilde{A}_{j,k,\ell}$ is the corresponding matrix (defined
in the obvious way), we have $A_{j,k,\ell}=\widetilde{A}_{j+k+1,k,\ell}$.
Hence, we can focus on finding a formula for
$\det(\widetilde{A}_{j,k,\ell})$. 

First of all, observe that
$\det(\widetilde{A}_{j,k,\ell})=(-1)^{k+\ell-2}\det(\widetilde{A}_{j-1,k,\ell})$
because $\widetilde{A}_{j,k,\ell}$ is obtained from
$\widetilde{A}_{j-1,k,\ell}$ by replacing the first
row by the sum of the first $k$ rows of the matrix,
and then permuting the rows so that the first row becomes the
last one. If we apply this idea repeatedly, we obtain that
$\det(\widetilde{A}_{j,k,\ell})=(-1)^{j(k+\ell-2)}\det(\widetilde{A}_{0,k,\ell})$.
Hence, it is sufficient to compute this latter determinant. 

We shall focus now on computing this determinant, which explicitly is
$$\det(\widetilde{A}_{0,k,\ell})=\det\begin{pmatrix}
\widetilde{C}_0^{(k,\ell)} & \widetilde{C}_{\ell}^{(k,\ell)} &
\widetilde{C}_{\ell+1}^{(k,\ell)} & \dots &
\widetilde{C}_{k+2\ell-3}^{(k,\ell)}\\ 
\widetilde{C}_{1}^{(k,\ell)} & \widetilde{C}_{\ell+1}^{(k,\ell)} &
\widetilde{C}_{\ell+2}^{(k,\ell)} & \dots &
\widetilde{C}_{k+2\ell-2}^{(k,\ell)}\\ 
\vdots & \vdots & \vdots & \ddots & \vdots\\
\widetilde{C}_{k+\ell-2}^{(k,\ell)} & \widetilde{C}_{k+2\ell-2}^{(k,\ell)} &
\widetilde{C}_{k+2\ell-1}^{(k,\ell)} & \dots &
\widetilde{C}_{2k+3\ell-5}^{(k,\ell)} 
\end{pmatrix}.$$
To do so, recall that we have
$\widetilde{C}_n^{(k,\ell)}=\displaystyle{\sum_{s=1}^{k+\ell-1}a_sr_s^n}$,
where the $a_i$'s are given by (\ref{aes}). We substitute this in the
above determinant and use multilinearity in the columns to expand it
into the sum 
$$\sum_{1\leq s_1,\dots,s_{k+\ell-1}\leq k+\ell-1}\left(\prod_{j=1}^{k+\ell-1}
a_{s_j}\right) \det_{1\leq i\leq k+\ell-1}
\left(r_{s_1}^{i-1}\ r_{s_2}^{i+\ell-1}\ r_{s_3}^{i+\ell}\ \cdots\ 
r_{s_{k+\ell-1}}^{i+k+2\ell-4}\right).$$    
Now, if in this sum two of the $s_j$'s should equal each other, then
the corresponding two columns in the determinant would be dependent so
that the determinant would vanish. We can therefore restrict the sum
to permutations of $\{1,2,\dots,k+\ell-1\}$. With $S_{k+\ell-1}$ denoting
the set of these permutations, this leads to 
{\allowdisplaybreaks
\begin{align}
\notag
\det(\widetilde{A}_{0,k,\ell})&=\sum_{\sigma\in
  S_{k+\ell-1}}\left(\prod_{j=1}^{k+\ell-1} a_{\sigma(j)}\right)\det_{1\leq
  i\leq k+\ell-1}
\left(r_{\sigma(1)}^{i-1}\ r_{\sigma(2)}^{i+\ell-1}\ r_{\sigma(3)}^{i+\ell}\ 
\cdots\ r_{\sigma(k+\ell-1)}^{i+k+2\ell-4}\right)\\ 
\notag 
&=\left(\prod_{j=1}^{k+\ell-1} a_{j}\right) \sum_{\sigma\in S_{k+\ell-1}}
\left(\prod_{j=2}^{k+\ell-1} r_{\sigma(j)}^{\ell+j-2}\right)\det_{1\leq
  i,j\leq k+\ell-1} \left(r_{\sigma(j)}^{i-1}\right)\\ 
\notag 
&=\left(\prod_{j=1}^{k+\ell-1} a_{j}\right) \sum_{\sigma\in S_{k+\ell-1}}
(\textrm{sgn}\ \sigma)\left(\prod_{j=2}^{k+\ell-1}
r_{\sigma(j)}^{\ell+j-2}\right)\det_{1\leq i,j\leq k+\ell-1}
\left(r_{j}^{i-1}\right)\\ 
\notag 
&=\left(\prod_{j=1}^{k+\ell-1} a_{j}\right)\left(\prod_{1\leq i<j\leq
  k+\ell-1} (r_j-r_i)\right)\sum_{\sigma\in S_{k+\ell-1}}
(\textrm{sgn}\ \sigma)\left(\prod_{j=2}^{k+\ell-1}
r_{\sigma(j)}^{\ell+j-2}\right)\\ 
\notag 
&=\left(\prod_{j=1}^{k+\ell-1} a_{j}\right)\left(\prod_{1\leq i<j\leq
  k+\ell-1} (r_j-r_i)\right)\det_{1\leq i\leq
  k+\ell-1}\left(1\ r_i^\ell\ r_i^{\ell+1}\ \cdots\ r_{i}^{k+2\ell-3}\right)\\ 
\notag 
&=\left(\prod_{j=1}^{k+\ell-1} a_{j}\right)\left(\prod_{1\leq i<j\leq
  k+\ell-1}
(r_j-r_i)\right)\left(\prod_{i=1}^{k+\ell-1}r_i\right)^{k+2\ell-3}\\ 
\notag 
&\hspace{2.5cm}
\times\det_{1\leq i\leq
  k+\ell-1}\left(r_i^{-k-2\ell+3}\ r_i^{-k-\ell+3}\ r_i^{-k-\ell+4}\ \cdots\ 1\right)\\ 
\notag 
&=\left(\prod_{j=1}^{k+\ell-1} a_{j}\right)\left(\prod_{1\leq i<j\leq
  k+\ell-1}
(r_j-r_i)(r_i^{-1}-r_j^{-1})\right)\left(\prod_{i=1}^{k+\ell-1}r_i\right)^{k+2\ell-3}\\ &\hspace{2.5cm}
\times h_{\ell-1}(r_1^{-1},\dots,r_{k+\ell-1}^{-1}). 
\label{prod}
\end{align}}%
In the last line we have used the following notations and facts: first
of all, $h_m(x_1,\dots, x_N)$ denotes the $m$-th complete homogeneous
symmetric function in $N$ variables $x_1,\dots, x_N$, explicitly given
by 
$$h_m(x_1,\dots,x_N)=\sum_{1\leq i_1\leq\cdots\leq i_m\leq N}
x_{i_1}\cdots x_{i_m}.$$ 
Furthermore, the Schur function indexed by a partition
$\lambda=(\lambda_1,\dots,\lambda_N)$ in the variables $x_1,\dots,x_N$
is defined by 
$$s_{\lambda}(x_1,\dots,x_N)=\frac{\displaystyle{\det_{1\leq i,j\leq
      N}\left(x_i^{\lambda_j+N-j}\right)}}{\displaystyle{\det_{1\leq
      i,j\leq
      N}\left(x_i^{N-j}\right)}}=\frac{\displaystyle{\det_{1\leq
      i,j\leq
      N}\left(x_i^{\lambda_j+N-j}\right)}}{\displaystyle{\prod_{1\leq
      i<j\leq N}(x_i-x_j)}}.$$ 
It is well-known (cf.\ \cite[p.~41, Eq.~(3.4)]{MAC}) that for
$\lambda=(m,0,\dots,0)$ the Schur function
$s_{\lambda}(x_1,\dots,x_N)$ reduces to $h_m(x_1,\dots,x_N)$. These
facts together explain the last line in the above computation. 

To proceed further, let us first observe that, by reading off the
constant coefficient of $g_{k,\ell}(x)$, we obtain 
$$\prod_{i=1}^{k+\ell-1} r_i=(-1)^{k+\ell}.$$
Furthermore, we have
\begin{align*}
g_{k,\ell}(x)&=x^{k+\ell-1}-\frac{x^k-1}{x-1}=\prod_{i=1}^{k+\ell-1}(x-r_i)
=(-1)^{k+\ell-1}\prod_{i=1}^{k+\ell-1}
r_i(1-r_i^{-1}x)\\ 
&=-\prod_{i=1}^{k+\ell-1}(1-r_i^{-1}x).
\end{align*}
Hence, we obtain
\begin{align*}
\sum_{m=0}^{\infty}h_m(r_1^{-1},\dots,r_{k+\ell-1}^{-1})x^m
&=\frac{1}{\displaystyle{\prod_{i=1}^{k+\ell-1}(1-r_i^{-1}x)}}\\  
&=\frac{1}{\frac{x^k-1}{x-1}-x^{k+\ell-1}}\\
&=\frac{1-x}{1-x^k-x^{k+\ell-1}+x^{k+\ell}}\\
&=1-x+x^k-x^{k+1}+\cdots+O(x^{k+\ell-1}).
\end{align*}
In order to evaluate $h_{\ell-1}(r_1^{-1},\dots,r_{k+\ell-1}^{-1})$, we just
have to extract the coefficient of $x^{\ell-1}$ in the expansion on the
right-hand side. This is easy: if $\ell-1$ equals a multiple of $k$ then
we obtain 1, if $\ell-2$ equals a multiple of $k$ then we obtain $-1$,
and in all other cases we obtain 0. 

We continue evaluating the other factors in (\ref{prod}). We have
\begin{align*}
\prod_{1\leq i<j\leq k+\ell-1} (r_j-r_i)(r_i^{-1}-r_j^{-1})&=\prod_{1\leq
  i<j\leq k+\ell-1}\frac{(r_j-r_i)^2}{r_ir_j}\\ 
&=\frac{\displaystyle{\prod_{1\leq i<j\leq
      k+\ell-1}(r_j-r_i)^2}}{\left(\displaystyle{\prod_{i=1}^{k+\ell-1}
    r_i}\right)^{k+\ell-2}}\\ 
&=\frac{\displaystyle{\prod_{1\leq i<j\leq k+\ell-1}(r_j-r_i)^2}}{(-1)^{k+\ell}}.
\end{align*}

Furthermore, we must compute $\displaystyle{\prod_{j=1}^{k+\ell-1}
  a_j}$. To begin with, recall the formula (\ref{aes}) and the fact that
$\widetilde{C}_0^{(k,\ell)}=\widetilde{C}_k^{(k,\ell)}=\dots
=\widetilde{C}_{k+\ell-2}^{(k,\ell)}=1$ and
$\widetilde{C}_1^{(k,\ell)}=\dots=\widetilde{C}_{k-1}^{(k,\ell)}=0$. 
With this in mind, we get 
\begin{align*}\prod_{j=1}^{k+\ell-1}
  a_j&=\frac{\displaystyle{\prod_{j=1}^{k+\ell-1}(-1)^{k+\ell+j-1}}}
{\displaystyle{\prod_{1\leq
        i<j\leq
        k+\ell-1}(r_j-r_i)^2}}\prod_{j=1}^{k+\ell-1}\left(\frac{r_j-1}{r_j(r_j-1)}
+\sum_{i=1}^{\ell-2}\frac{r_j^k-1}{r_j^{k+i}(r_j-1)}+1\right)\\  
&=\frac{(-1)^{\frac{(3k+3\ell-2)(k+\ell-1)}{2}}}{\displaystyle{\prod_{1\leq
      i<j\leq
      k+\ell-1}(r_j-r_i)^2}}\prod_{j=1}^{k+\ell-1}\left(\frac {1} {r_j}
+\sum_{i=1}^{\ell-2}\frac{r_j^k-1}{r_j^{k+i}(r_j-1)}+1\right).
\end{align*} 
Moreover, observe that
\begin{align*}
\frac {1} {r_j}
+\sum_{i=1}^{\ell-2}\frac{r_j^k-1}{r_j^{k+i}(r_j-1)}+1
&=\frac{1}{r_j}+\frac{r_j^k-1}{r_j-1}\sum_{i=1}^{\ell-2}\frac{1}{r_j^{k+i}}+1\\ 
&=\frac{1}{r_j}+r_j^{k+\ell-1}\sum_{i=1}^{\ell-2}\frac{1}{r_j^{k+i}}+1\\
&=\frac{r_j^\ell-1}{r_j(r_j-1)}.
\end{align*}
Here, to obtain the second line, 
we have used the fact that $1\neq r_j$ is a root of
$x^{k+\ell-1}-\frac{x^k-1}{x-1}$. 

Now, to conclude we must compute
$\displaystyle{\prod_{j=1}^{k+\ell-1}\frac{r_j^\ell-1}{r_j(r_j-1)}}$. To do
so, let $\omega$ be a primitive $\ell$-th root of unity. Then 
\begin{align*}
\prod_{j=1}^{k+\ell-1}(r_j^\ell-1)&=\prod_{j=1}^{k+\ell-1}
\prod_{i=1}^{\ell}(r_j-\omega^{i})=\prod_{i=1}^{\ell}
\prod_{j=1}^{k+\ell-1}(r_j-\omega^{i})\\ 
&=\left(\prod_{j=1}^{k+\ell-1}(r_j-1)\right)\left(\prod_{i=1}^{\ell-1}
\prod_{j=1}^{k+\ell-1}(r_j-\omega^{i})\right)\\
&=\left(\prod_{j=1}^{k+\ell-1}(r_j-1)\right)(-1)^{(k+\ell-1)(\ell-1)}
\left(\prod_{i=1}^{\ell-1}g_{k,\ell}(\omega^{i})\right).
\end{align*}
Furthermore,
$g_{k,\ell}(\omega^{i})=\omega^{i(k+\ell-1)}-\frac{\omega^{ik}-1}{\omega^{i}-1}=-\frac{\omega^{i(k-1)}-1}{\omega^{i}-1}$. Consequently,
we have 
\begin{align*}
\prod_{j=1}^{k+\ell-1}\frac{r_j^\ell-1}{r_j(r_j-1)}&=(-1)^{(k+\ell)\ell}
\left(\prod_{i=1}^{\ell-1}\frac{\omega^{i(k-1)}-1}{\omega^{i}-1}\right). 
\end{align*}
Finally observe that
$$\prod_{i=1}^{\ell-1}\frac{\omega^{i(k-1)}-1}{\omega^{i}-1}=\begin{cases}
1, & \textrm{if $\gcd(\ell,k-1)=1$};\\
0, & \textrm{otherwise}.
\end{cases}
$$

We can now collect all the work done to obtain the following result.

\begin{teor}
For all integers $k$ and $\ell$ with $k,\ell\ge2$, we have
$$\det(\widetilde{A}_{0,k,\ell})=\begin{cases}
(-1)^{\frac{(k+\ell)(k+\ell-1)}{2}+1}, & \textrm{if $\ell-1=\alpha k$ and
    $\gcd(\ell,k-1)=1$};\\ 
(-1)^{\frac{(k+\ell)(k+\ell-1)}{2}}, & \textrm{if $\ell-2=\beta k$ and
    $\gcd(\ell,k-1)=1$};\\ 
0, & \textrm{otherwise}.
\end{cases}
$$
\end{teor}

\begin{coro}
Let $k_0,\ell_0\ge2$ be any integers. Then the following hold:
\begin{itemize}
\item[i)] The sequence $\{\alpha_\ell\}_{\ell\geq 2}$
given by $\alpha_\ell=|\det(\widetilde{A}_{0,k_0,\ell})|$ is
periodic, and its period is a divisor of $k_0\cdot\rad(k_0-1)$. 
\item[ii)] The sequence $\{\beta_k\}_{k\geq k}$ given by $\beta_k=|\det(\widetilde{A}_{0,k,\ell_0})|$ is eventually zero.
\end{itemize}
\end{coro}
\begin{proof}
\begin{itemize}
\item[i)] Clearly $\gcd(\ell,k_0-1)>1$ implies that
$\gcd(\ell+k_0\cdot\rad(k_0-1),k_0-1)>1$. In the same way, if $\ell-1$ and
$\ell-2$ are not multiples of $k_0$, then neither are $\ell+k_0\cdot\rad(k_0-1)-1$
or $\ell+k_0\cdot\rad(k_0-1)-2$. Consequently, if $\alpha_\ell=0$, also
$\alpha_{\ell+k_0\cdot\rad(k_0-1)}=0$ as claimed. 
\item[ii)] If $k\geq \ell_0$ obviously neither $\ell-1$ nor $\ell-2$ can be multiples of $k$ and therefore $\beta_{k}=0$ for every $k\geq \ell_0$.
\end{itemize}
\end{proof}

\end{document}